\documentclass{amsart}
\usepackage{amsmath,amssymb,latexsym, amsfonts, amscd, amsthm}

\usepackage{hyperref}
\hypersetup{colorlinks=false}
\usepackage{graphicx}

\usepackage{verbatim}
\usepackage{xcolor}
\hypersetup{colorlinks=false,linkbordercolor=red,linkcolor=green,pdfborderstyle={/S/U/W 1}}

\usepackage{bbm}
\usepackage{enumerate}
\usepackage[all]{xy}

\newtheorem{thm}{Theorem}[section]
\newtheorem{pro}[thm]{Proposition}
\newtheorem{lm}[thm]{Lemma}
\newtheorem{cor}[thm]{Corollary}

\numberwithin{equation}{section}
\newtheorem{hyp}[thm]{Working Hypothesis}

\theoremstyle{definition}

\theoremstyle{remark}
\newtheorem{rem}[thm]{Remark}

\numberwithin{equation}{section}

\DeclareMathOperator*{\el}{ell}
\DeclareMathOperator*{\Irr}{Irr}
\DeclareMathOperator*{\disc}{disc}
\DeclareMathOperator*{\Gal}{Gal}
\DeclareMathOperator*{\Res}{Res}

\DeclareMathOperator*{\temp}{temp}
\DeclareMathOperator*{\ad}{ad}
\DeclareMathOperator*{\scn}{sc}

\DeclareMathOperator*{\der}{der}

\DeclareMathOperator*{\Hom}{Hom}

\newcommand{\vp}{\varphi}

\newcommand{\s}{\simeq}

\newcommand{\si}{\sigma}

\newcommand{\ts}{\sigma^{\sharp}}

\newcommand{\teta}{\eta^\sharp}

\newcommand{\tchi}{\chi^\sharp}

\newcommand{\tphi}{\phi^\sharp}

\newcommand{\CC}{\mathbb{C}}
\newcommand{\NN}{\mathbb{N}}

\newcommand{\ii}{\mathbf{\textit{i}}}

\def\bA{\bold A}
\def\bG{\bold G}

\def\bP{\bold P}
\def\bM{\bold M}
\def\bN{\bold N}

\def\bA{\bold A}
\def\bG{\bold G}
\def\bP{\bold P}
\def\bM{\bold M}
\def\bN{\bold N}

\def\bJ{\bold J}

\def\L{\mathcal L}

\DeclareMathOperator*{\Sp}{Sp}

\DeclareMathOperator*{\SU}{SU}

\DeclareMathOperator*{\U}{U}
\DeclareMathOperator*{\SO}{SO}

\DeclareMathOperator*{\SL}{SL}
\DeclareMathOperator*{\GL}{GL}

\DeclareMathOperator*{\GSO}{GSO}
\DeclareMathOperator*{\GSp}{GSp}
\DeclareMathOperator*{\Spin}{Spin}

\DeclareMathOperator*{\GSpin}{GSpin}

\newcommand{\tG}{G^{\sharp}}
\newcommand{\tbG}{\bold G^{\sharp}}
\newcommand{\tP}{P^\sharp}

\newcommand{\tM}{M^\sharp}

\newcommand{\tbM}{\bold M^\sharp}

\def\L{\mathcal L}

\begin{document}

\title[On compatibility of Arthur's conjecture for $R$-groups]{On compatibility in restriction of Arthur's conjecture for $R$-groups}

\author[Kwangho Choiy]{Kwangho Choiy} 
\address{School of Mathematical and Statistical Sciences,
Southern Illinois University,
Carbondale, IL 62901-4408,
U.S.A.}
\email{kchoiy@siu.edu}

\subjclass[2020]{Primary \textbf{22E50}; Secondary 11F70, 22E55, 22E35}

\begin{abstract}
We study the compatibility of Arthur's conjecture for $R$-groups in the restriction of discrete series representations from Levi subgroups of a $p$-adic group to those of its closed subgroup having the same derived group. The compatibility is conditional as it holds under the conjectural local Langlands correspondence in general. This work is applied to several cases and further suggests a uniform way to verify the Arthur's conjecture for $R$-groups in the setting.
\end{abstract}
%%%%%%%%%%%%%%%%%%%%%%%%%%%%%%%%%%%%%%%%%%%%%%%%%
\maketitle

%%%%%%%%%%%%%%%%%%%%% 
\section{Introduction} \label{intro}
%%%%%%%%%%%%%%%%%%%%%%%%%%%%%%%%%%%%%

Arthur's conjecture for $R$-groups proposes an isomorphism between the Knapp-Stein and Langlands-Arthur $R$-groups in \cite{art89ast}. It has been proved for many cases over some decades and contributes to comparison of trace formulae and to the endoscopic classification of automorphic representations \cite{art12, kmsw14, mok13}.

The purpose of the paper is to investigate the compatibility of Arthur's conjecture for $R$-groups, when one restricts discrete series representations from Levi subgroups of a connected reductive algebraic group over a $p$-adic field to those of its closed subgroup having the same derived group. Our result is conditional on the conjectural local Langlands correspondence in general and a working hypothesis.  Along with applications to cases when those assumptions are verified, the compatibility further suggests a uniform way to verify the Arthur's conjecture for $R$-groups in the setting. To this end, we shall verify several group structural statements in both $p$-adic and $L$-group sides. Alongside properties of Plancherel measures, we also study stabilizers in Weyl group of representations and $L$-parameters, and centralizers and normalizers of $L$-parameters. We then related them to each other under some working hypotheses including the local Langlands correspondences in general. Certain finite abelian groups of continuous characters, and Galois cohomology groups will work together in some steps and in several exact sequences consisting of groups produced over $p$-adic groups and their $L$-groups.

To be precise, we let $\tbG$ be a connected reductive group over a $p$-adic field $F,$ 
and let $\bG$ be a closed $F$-subgroup of $\tbG$ with the following  property that the inclusions of algebraic groups
\begin{equation} \label{cond on G intro}
\bG_{\der} = \tbG_{\der} \subseteq \bG \subseteq \tbG
\end{equation}  
hold, where the subscript ${\der}$ stands for the derived group. Then we have the following exact sequence
\[
1 \longrightarrow \widehat{\tG/G} \longrightarrow \widehat \tG \longrightarrow \widehat G  \longrightarrow  1,
\]
where 
$\widehat G, \widehat \tG,  \widehat{\tG/G}$ are denoted by the connected components of the $L$-group of $\bG, \tbG/\bG, \tbG$ (see Section \ref{llc} for the details), and $\widehat{\tG/G}$ is a central torus in $\widehat \tG$ (refer to Remark \ref{rem Labesse situation}). 

Let $\bM$ be an $F$-Levi subgroup of $\bG,$ and let $\tbM$ be an $F$-Levi subgroup of $\tbG$ such that $\tbM = \bM \cap \bG.$ Given $\phi \in \Phi(M),$ there is a lifting $\tphi \in \Phi(\tM)$ such that
\begin{equation} \label{projection}
\phi = pr \circ \tphi,
\end{equation}
where $pr$ is the projection  $\widehat{\tM} \twoheadrightarrow \widehat{M}.$ 

This is due to Remark \ref{rem Labesse situation} and \cite[Th\'{e}or\`{e}m 8.1]{la85}, and the further details are given in Section \ref{a key proposition}.

Now, we shall state our assumption (Working Hypothesis \ref{workhyp}) with a connected reductive group $\boldsymbol{J}$ over $F$ and its $F$-Levi subgroup $\bM_{\bJ}.$
Under the validity of the local Langlands correspondence for tempered representations of $J$ and for discrete series representations of $M_J,$ given $\phi_J \in \Phi_{\disc}(M_J)$ and $\sigma_J \in \Pi_{\phi}(M_J),$ we have
\begin{equation} \label{whp equ}
W_{\sigma_J}^\circ = W_{\sigma_J, \phi_J} ^\circ,
\end{equation}
where $W_{\phi_J}$ is considered as a subgroup of $W_{M_J}.$ We refer the reader to relevant definitions in Sections \ref{llc} and \ref{rgp arthur conj}.

Given $\ts \in \Pi_{\tphi}(\tM)$ and $\sigma \in \Pi_{\phi}(M),$ under the validity of the hypotheses for $\bJ=\tbG,$
we have the following (Theorem \ref{thm from tG})
\begin{equation}  \label{arthur conj for tG intro}
R_{\ts} \s R_{\tphi, \ts}, ~~ \text{ and } ~~ R_{\sigma} \s R_{\phi, \sigma}.
\end{equation}

On the other hand, we let $\phi \in \Phi_{\disc}(M)$ be given.  Choose a lifting $\tphi \in \Phi_{\disc}(\tM)$ such that $\phi = pr \circ \tphi,$ with the projection $pr: \widehat{\tM} \twoheadrightarrow \widehat{M}.$
Given $\ts \in \Pi_{\tphi}(\tM)$ and $\sigma \in \Pi_{\phi}(M),$ under the validity of the hypotheses for $\bJ=\bG,$ we have those isomorphism \eqref{arthur conj for tG intro} (Theorem \ref{thm from G}).

Two key statements (Propositions \ref{lm for identity comp} and \ref{a key prop}) among others that are utilized to prove \eqref{arthur conj for tG intro} under the hypothesis for both $\bJ=\bG$ and $\bJ=\tbG$ are
\[
\bar C_{\tphi}^\circ =  C_{\phi_{\scn}}(\widehat{G_{\scn}})^\circ = \bar C_{\phi}^\circ ~~ \text{ and } ~~
W^{\circ}_{\tphi} = W^{\circ}_{\phi}.
\]
These essentially guide us in working in the adjoint group of their $L$-groups. 
Using the given condition \eqref{cond on G intro}, those equalities above stem mainly from studies of several relationships between $\tbG$ and $\bG$ and between their $L$-groups, as well as those between  $\tbM$ and $\bM$ and between their $L$-groups.  Moreover, it is crucial to use the following finite group
\begin{equation*} \label{def of X intro}
X^{\tG}(\tphi) := \{ \bold a \in H^1(W_F,\widehat{\tG/G}) : \bold a \tphi \s \tphi  \},
\end{equation*}
which lies in the following exact sequence
\[
1 \longrightarrow C_{\tphi}(\widehat{\tG})/Z(\widehat{\tG})^\Gamma \longrightarrow C_{\phi_{\scn}}(\widehat{G_{\scn}}) \longrightarrow X^{\tG}(\tphi) \longrightarrow 1
\]
(see (\eqref{def of X}).
Using the property of Plancherel measures,
we have
\begin{equation*} \label{more intro}
W^{\circ}_{\ts} = W^{\circ}_{\sigma} = W^{\circ}_{\tphi, \ts}.
\end{equation*}
Due to the inclusion $W_{\tphi} \cap W(\sigma) \subset W(\ts)$ (Lemma \ref{lm imp1}), the above key group structural statements, and definitions, we have 
\[
W^\circ_{\phi, \sigma} = W^\circ_\sigma
\]
(see \eqref{equal for si}).
This is combined with the equality $W(\sigma) = W_{\phi, \sigma}$ (\eqref{reduction for W}), which comes from the inclusion $W(\sigma) \subset W_\phi$ (Lemma \ref{a inclusion on W}). 
This proves the case $R_{\sigma} \s R_{\phi, \sigma}$ in \eqref{arthur conj for tG intro} from the hypothesis above for $\bJ=\tbG.$ The other cases are verified in a similar way. The full detail is given in Sections \ref{sec arthur conj from tG} and \ref{sec arthur conj from G}.

We remark that the method turns out to rely on group structural arguments and properties of Plancherel measures, both of which are \textit{unconditional}. However, the hypothesis above is required as we need to transfer arguments between representations and $L$-parameters. This machinery generalizes an idea in \cite{cgsu} for the case that $\tbG=\U_n$ and $\bG=\SU_n$ and their non-quasi-split $F$-inner forms,
and then it is now applied to many cases such as $(\tbG,\bG)=(\GL_n,\SL_n),$ $(\GSp_{2n}, \Sp_{2n}),$ $(\GSO_n, \SO_n),$ $(\GSpin_n, \Spin_n),$ and their $F$-inner forms (refer to Section \ref{sec clos rem}).

In Section \ref{prelim}, we recall basic notation and background, discuss local Langlands conjectures for tempered representations in general, and review three $R$-groups and Arthur's conjecture. In Section \ref{main thm section}, we prove Arthur's conjectures for $R$-groups for both $\tG$ and $G,$ in two directions: one is under the working hypothesis for $\bJ=\tbG;$ and the other is under the working hypothesis for $\bJ=\bG.$ Also, several group structural arguments in both $p$-adic and $L$-group sides are discussed.

%%%%%%%%%%%%%%%%%%%%%%%%%%%%%%%%

%%%%%%%%%%%%%%%%%%%%%%%%%%%%%%%%%
\section{Preliminaries} \label{prelim}
%%%%%%%%%%%%%%%%%%
Throughout the paper, we denote by $F$ a $p$-adic field of characteristic $0$ with an algebraic closure $\bar{F}.$  By $W_F$ we denote the Weil group of $F$ and by $\Gamma$ the absolute Galois group $\Gal(\bar{F} / F).$ 
%%%%%%%%%%%%%%%
\subsection{Notation and basic backgrounds} \label{notation}
%%%%%%%%%%%%%%%%%%%%%%%%%%%%%%%%%
Let $\bG$ be a connected reductive algebraic group defined over $F.$ We write $G$ for the group $\bG(F)$ of $F$-points and for other algebraic groups defined over $F.$
Fix a minimal $F$-parabolic subgroup $\bP_0$ of $\bG$ with Levi decomposition $\bP_0=\bM_0 \bN_0.$ 
We denote by $\bA_0$ the split component of $\bM_0$ and 
by $\Delta$ the set of simple roots of $\bA_0$ in $\bN_0.$ 
Let $\bP$ be an $F$-parabolic subgroup with Levi decomposition $\bP=\bM\bN$ such that $\bM \supseteq \bM_0$ and $\bN \subseteq \bN_0.$
We note that there is a subset $\Theta \subseteq \Delta$ such that $\bM$ equals the Levi subgroup $\bM_{\Theta}$ generated by $\Theta.$
Set $\bA_{\bM_{\Theta}}=\bA_{\bM}$ for the split component of $\bM=\bM_{\Theta}.$  
We write $\Phi(P, A_M)$ for the set of reduced roots of $\bP$ with respect to $\bA_\bM.$  Set  $W_M = W(\bG, \bA_\bM) := N_\bG(\bA_\bM) / Z_\bG(\bA_\bM)$ the Weyl group of $\bA_\bM$ in $\bG,$ 
where, $N_\bG(\bA_\bM)$ and $Z_\bG(\bA_\bM)$ are the normalizer and centralizer of $\bA_\bM$ in $\bG,$ respectively. 

The set of isomorphism classes of irreducible admissible complex representations of $G$ is denoted by $\Irr(G).$ For simplicity, we do not distinguish each isomorphism class from its representative. 
Let $\sigma \in \Irr(M)$ be given. Write $\ii_{G,M} (\sigma)$ for the normalized (twisted by $\delta_{P}^{1/2}$) induced representation, where $\delta_P$ denotes the modulus character of $P.$ 
Denote by $\sigma^{\vee}$ the contragredient of $\sigma.$
The subset consisting of discrete series (respectively, tempered representations) in $\Irr(G)$ are denoted by $\Pi_{\disc}(G)$ (respectively, $\Pi_{\temp}(G)$).  
Here, a discrete series representation stands for an irreducible, admissible, unitary representation whose matrix coefficients are square-integrable modulo the center of $G,$ that is, in $L^2(G/Z(G)),$
and a tempered representation means an irreducible, admissible, unitary representation whose matrix coefficients are in $L^{2+\epsilon}(G/Z(G))$ for all $\epsilon > 0.$

Given any topological group $H$ and its subset $I,$ we denote by $Z(H)$ the center of $H$ and by $Z_H(I)$ (resp., $N_H(I)$) the centralizer (resp., normalizer) of $I$ in $H.$
We denote by $H^\circ$ the identity component of $H,$ and by $\pi_0(H)$ the group $H/H^\circ$ of connected components of $H.$ 
By $H_{\ad}=H/Z(H)$ we denote the adjoint group  of $H,$ and by $H_{\scn}$ the simply-connected cover of the derived group $H_{\der}$ of $H.$
Set $H^D:=\Hom(H , \CC^{\times}),$ the group of all continuous characters.
Given a Galois module $J,$ for $i \in \NN$ we set $H^i(F, J) := H^i(\Gal (\bar{F} / F), J(\bar{F})),$  
the Galois cohomology of $J.$

%%%%%%%%%
\subsection{Local Langlands conjecture for tempered representations in a general setting} \label{llc}
%%%%%%%%%%
We follow \cite[Section 2]{bo79} to state the local Langlands conjecture for tempered representations in a general setting. 
Let $\bG$ be a connected reductive algebraic group over $F.$ 
Fixing $\Gamma$-invariant splitting data, we define the $L$-group of $G$ as a semi-direct product $^{L}G := \widehat{G} \rtimes W_F.$ 
As in Section 8.2 of \textit{loc. cit.}, we say an $L$-parameter for $G$ is an admissible homomorphism 
\[
\vp: W_F \times SL_2(\CC) \rightarrow {^L}G.
\]
Note that two $L$-parameters are said to be equivalent if they are conjugate by $\widehat{G}.$
The set of equivalence
classes of $L$-parameters for $G$ is denoted by $\Phi(G).$
An $L$-parameter $\vp$ is said to be tempered if $\vp(W_F)$ is bounded, and we denote by $\Phi_{\temp}(G)$ the subset of $\Phi(G)$ which consist of tempered $L$-parameters of $G.$
The local Langlands conjecture for tempered representations of $G$ predicts that there is a surjective finite-to-one map 
\begin{equation} \label{pre temp llc}
\L:{\Irr}_{\temp}(G) \rightarrow \Phi_{\temp}(G).
\end{equation}
Let $\vp \in \Phi_{\temp}(G)$ be given. 
By $\Pi_{\vp}(G):=\L^{-1}(\vp)$ we denote the $L$-packet attached to $\vp.$ 
The map \eqref{pre temp llc} implies that
\begin{equation} \label{temp llc}
{\Irr}_{\temp}(G) = \bigsqcup_{\vp \in {\Phi}_{\temp}(G)}\Pi_{\vp}(G).
\end{equation}

Given $\vp \in \Phi_(G),$ we write $C_{\vp}(\widehat{G})$ for the centralizer of the image of $\vp$ in $\widehat{G}.$
We note that $C_{\vp}$ contains the center $Z(\widehat{G})^{\Gamma}$  of $^{L}G.$
An $L$-parameter $\vp$ is said to be elliptic if the quotient group $C_{\vp}(\widehat G) / Z(\widehat{G})^{\Gamma}$
is finite.
The subset in $\Phi(G)$ consisting of elliptic $L$-parameters is denoted by $\Phi_{\el}(G).$ Set $\Phi_{\disc}(G) = \Phi_{\el}(G) \cap \Phi_{\temp}(G).$
Then, the map \eqref{pre temp llc} and disjoint union \eqref{temp llc} both hold between $\Pi_{\disc}(G)$ and $\Pi_{\temp}(G).$
%%%%%%%%%%%%%%%%%%%%%%%%%%%%%%%%%%
%%%%%%%%%%%%%%%%%%%%%%%%%%%%%%%%
\subsection{Three $R$-groups and Arthur's conjecture} \label{rgp arthur conj}
%%%%%%%%%%%%%%%%%%%%%%%%%%%%%
Following \cite{art89ast, goldberg-class, ks72, sil78, cgsu}, we recall Knapp-Stein, Langlands-Arthur, and endoscopic $R$--groups, and discuss Arthur's conjecture for $R$-groups. 

Let $\bG$ be a connected reductive algebraic group over $F$ and $\bM$ an $F$-Levi subgroup of $\bG.$
For $\sigma \in \Pi_{\disc}(M)$ and $w \in W_M,$ 
the representation ${^w}\sigma$ twisted by $w$ is given by 
\[
{^w}\sigma(x)=\sigma(w^{-1}xw).
\]
We note that the isomorphism class of ${^w}\sigma$ is independent of the choices of representatives in $G$ of $w \in W_M,$ and  do not distinguish an element in the normalizer $N_\bG(\bA_\bM)$ and its representative in $W_M.$ 
Let $\sigma \in \Pi_{\disc}(M)$ be given. By $W(\sigma)$ we denote the stabilizer 
\[
\{ w \in W_M : {^w}\sigma \s \sigma \}
\]
of $\sigma$ in $W_M.$
We write $\Delta'_\sigma$ for the subset of $\Phi(P, A_M)$ consisting of $\alpha$ such that $\mu_{\alpha} (\sigma) = 0 \},$ where $\mu_{\alpha} (\sigma)$ is the rank one Plancherel measure for $\sigma$ attached to $\alpha$ \cite[p.1108]{goldberg-class}. 

We  define \textit{Knapp-Stein $R$-group} $R_{\sigma}$ to be
\[
\{ w \in W(\sigma) : w \alpha > 0, \; \forall \alpha \in \Delta'_\sigma \}.
\]
The subgroup of $W(\sigma)$ generated by the reflections in the roots of $\Delta'_\sigma$ is denoted by  $W^{\circ}_{\sigma}.$ 
It then follows that
\[
W(\sigma) = R_\sigma \ltimes W^{\circ}_{\sigma},
\]
and in turn we have
\[
R_\sigma \s  W(\sigma)/W^{\circ}_{\sigma}.
\]

Now, via the inclusion $\widehat M \hookrightarrow \widehat G,$
it is natural that an $L$-parameter $\phi \in \Phi_{\disc}(M)$ is also considered as an $L$-parameter for $G.$ 
Fixing a maximal torus $T_{\phi}$ in $C_{\phi}(\widehat{G})^{\circ},$ 
we set
\begin{eqnarray} 
\nonumber 
W_{\phi}^{\circ} &:=& N_{C_{\phi}(\widehat{G})^{\circ}} (T_{\phi}) /  Z_{C_{\phi}(\widehat{G})^{\circ}} (T_{\phi}), \\
\nonumber 
 \quad W_{\phi} &:=& N_{C_{\phi}(\widehat{G})} (T_{\phi}) /  Z_{C_{\phi}(\widehat{G})} (T_{\phi}).
\end{eqnarray} 

We  define the \textit{endoscopic $R$-group} $R_{\phi}$ as the quotient
\[
W_{\phi}/W_{\phi}^{\circ}.
\]
Identifying $W_{\phi}$ with a subgroup of $W_M$ (see \cite[p.45]{art89ast}), 
for $\sigma \in \Pi_{\phi}(M),$ we set
\begin{eqnarray*} \label{def of W_phi, sigma}
W_{\phi, \sigma}^{\circ} &:=&  W_{\phi}^{\circ} \cap W(\sigma),  \\
\quad  W_{\phi, \sigma} &:=&  W_{\phi} \cap W(\sigma).
\end{eqnarray*} 

We  define the \textit{Langlands-Arthur $R$-group} $R_{\phi, \sigma}$ to be the quotient
\[
W_{\phi, \sigma}/W_{\phi, \sigma}^{\circ}.
\]
It follows from \cite[p.46]{art89ast} that $R_{\phi, \sigma} \subset R_{\phi}.$

Arthur's conjecture for $R$-groups predicts that given a discrete series representation $\sigma \in \Pi_{\phi}(M),$ we have
\begin{equation} \label{art conj}
R_{\sigma} \s R_{\phi, \sigma}.
\end{equation}
The conjecture has been proved for various cases by Arthur, Ban-Goldberg, Ban-Zhang, Choiy-Goldberg, Mok, and others \cite{art12, bg12, bggspin,  baj05,  chaoli, cgsu, chgo12, cgclassic, mok13}.

%%%%%%%%%%%%%%%%%%%%%%%%%%%%%%%%%%%%%%%%%%\\
\section{Main theorems - Compatibility of Arthur's conjecture for $R$-groups in restriction} \label{main thm section}

This section is devoted to presenting our main theorems on the compatibility of Arthur's conjecture for $R$-groups in restriction. Throughout the section, we let $\tbG$ be a connected reductive group over $F,$ 
and let $\bG$ be a closed $F$-subgroup of $\tbG$ with the following property that the inclusions of algebraic groups
\begin{equation}  \label{cond on G}
\bG_{\der} = \tbG_{\der} \subseteq \bG \subseteq \tbG
\end{equation}  
hold. 
Here, the subscript ${\der}$ stands for the derived group.
\begin{rem} \label{rem Labesse situation}
We note that \eqref{cond on G} implies the following exact sequence
\[
1 \longrightarrow \widehat{\tG/G} \longrightarrow Z(\widehat \tG)  \longrightarrow Z(\widehat G)  \longrightarrow  1,
\]
due to \cite[(1.8.1)]{kot84}, and
\[
1 \longrightarrow \widehat{\tG/G} \longrightarrow \widehat \tG \longrightarrow \widehat G  \longrightarrow  1,
\]
referring to \cite[p.58]{chaoli} (note that their notation is different from ours).
It then follows that $ \widehat{\tG/G}$ is a central torus in $\widehat \tG.$ Thus, we have the following commutative diagram
\begin{equation} \label{important comm diag for G}
\begin{CD}
@.  1 @. 1 @.  @.
\\
@.      @VVV        @VVV   @. @.\\
1 @>>>  \widehat{\tG/G}@>{\s}>> \ker @>>> 1 @. @. 
\\
@.      @VVV        @VVV   @VVV  @.\\
1 @>>> Z(\widehat{\tG}) @>>> \widehat{\tG} @>>> (\widehat{\tG})_{ad} @>>> 1 
\\
@.      @VVV          @VVV   @|  @.
\\
1 @>>> Z(\widehat{G})  @>>> \widehat{G} @>>> (\widehat{G})_{ad}@>>> 1 \\
@.      @VVV          @VVV   @VVV  @.\\
@.  1 @. 1 @. 1 @. 
\end{CD}
\end{equation}
The condition \eqref{cond on G} is satisfied in many cases such as $(\tbG,\bG)=(\GL_n,\SL_n),$ $(\GSp_{2n}, \Sp_{2n}),$ $(\GSO_n, \SO_n),$ $(\GSpin_n, \Spin_n),$ $(\U_n, \SU_n)$ and their $F$-inner forms. 
\end{rem}

We will prove Arthur's conjectures for $R$-groups for both $\tG$ and $G,$ provided that Working Hypothesis \ref{workhyp} for $\bJ=\tbG$ is valid (Theorem \ref{thm from tG}). We will also verify the same statement, provided that Working Hypothesis \ref{workhyp} for $\bJ=\bG$ is valid (Theorem \ref{thm from G}).
The method in the proof relies on group structural arguments in $L$-groups in Section \ref{a key proposition} and properties of Plancherel measures, both of which are \textit{unconditional}. However, the hypothesis is required as we need to transfer arguments between representations and $L$-parameters. This machinery generalizes an idea in \cite{cgsu} for the case that $\tbG=\U_n$ and $\bG=\SU_n$ with respect to a quadratic extension $E$ over $F,$ and their non-quasi-split $F$-inner forms.

\subsection{Group structural argument in $L$-groups} \label{a key proposition}
%%%%%%%%%%%
Let $\bM$ be an $F$-Levi subgroup of $\bG,$ and let $\tbM$ be an $F$-Levi subgroup of $\tbG$ such that $\tbM = \bM \cap \bG.$
We note from \cite[(1.8.1) p.616]{kot84} that the exact sequence of algebraic groups
\begin{equation} \label{lgps for M}
1 \longrightarrow \bM \longrightarrow \tbM \longrightarrow \tbM/\bM \longrightarrow 1.
\end{equation}
\begin{rem} \label{twoquotients}
The inclusion $\bM \hookrightarrow \bG$ induces an isomorphism as $F$-tori
\[
\tbM/\bM \s \tbG/\bG,
\]
cf., \cite[Proposition 3.2.3]{chaoli} (note that their $~^\sharp$ indicates smaller groups).
\end{rem}
\begin{rem}
Due to Remark \ref{rem Labesse situation} and \eqref{lgps for M}, we have
an exact sequence
\begin{equation} \label{exact of L centers}
1 \longrightarrow  \widehat{\tM/M} \longrightarrow Z({\widehat{\tM}}) \longrightarrow Z({\widehat{M}}) \longrightarrow 1.
\end{equation}
We thus have the following commutative diagram of $L$-groups which is analogue of \eqref{important comm diag for G} (cf.,  and \cite[Remark 2.4]{chgo12})
\begin{equation} \label{important comm diag}
\begin{CD}
@.  1 @. 1 @.  @.
\\
@.      @VVV        @VVV   @. @.\\
1 @>>>  \widehat{\tM/M}@>{\s}>> \ker @>>> 1 @. @. 
\\
@.      @VVV        @VVV   @VVV  @.\\
1 @>>> Z(\widehat{\tM}) @>>> \widehat{\tM} @>>> (\widehat{\tM})_{ad} @>>> 1 
\\
@.      @VVV          @VVV   @|  @.
\\
1 @>>> Z(\widehat{M})  @>>> \widehat{M} @>>> (\widehat{M})_{ad}@>>> 1 \\
@.      @VVV          @VVV   @VVV  @.\\
@.  1 @. 1 @. 1 @. 
\end{CD}
\end{equation}
\end{rem}
Given $\phi \in \Phi(M),$ by Remark \ref{rem Labesse situation} and \cite[Th\'{e}or\`{e}m 8.1]{la85}, there is a lifting $\tphi \in \Phi(\tM)$ such that
\begin{equation} \label{projection}
\phi = pr \circ \tphi,
\end{equation}
where $pr$ is the projection  $\widehat{\tM} \twoheadrightarrow \widehat{M},$ cf., the middle vertical exact sequence in \eqref{important comm diag}.
Since the homomorphism $pr$ is compatible with $\Gamma$-actions on $\widehat{\tM}$ and $\widehat{M},$ the lifting $\tphi \in \Phi_{\disc}(\tM)$ is chosen uniquely up to a $1$-cocycle of $W_F$ in $\widehat{(\tM/M)}$ due to \cite[Section 7]{la85}, and if $\phi \in \Phi_{\disc}(M),$ the lifting $\tphi$ lies in $\Phi_{\disc}(\tM).$

From \cite[Chapter 3.4]{mok13} we set a maximal torus $T_{\tphi}$ in $C_{\tphi}(\widehat{\tG})^{\circ}$ to be the identity component
\[
\bA_{\widehat \tM} = (Z(\widehat{\tM})^{\Gamma})^{\circ}
\]
of the $\Gamma$--invariants of the center $Z(\widehat{\tM}).$ 
So, we have $\widehat \tM = Z_{\widehat \tG}(T_{\tphi})$ (cf., \cite[Lemma 3.5 and Propositions 3.6 \& 8.6]{bo79}).  
Likewise, we set $T_{\phi}= A_{\widehat M} \subset C_{\phi}(\widehat{G})^{\circ},$ and  $\widehat M = Z_{\widehat G}(T_{\phi}).$ We now write
\begin{eqnarray} \label{quotients}
\bar C_{\tphi}^\circ &:=& C_{\tphi}(\widehat{\tG})^\circ/Z(\widehat \tG)^{\Gamma}, \\
\nonumber 
\bar T_{\tphi} &:=& T_{\tphi}/Z(\widehat \tG)^{\Gamma}, \\
\nonumber 
\bar C_{\phi}^\circ &:=& C_{\phi}(\widehat{G})^\circ/Z(\widehat G)^{\Gamma}, \\
\nonumber 
\bar T_{\phi} &:=& T_{\phi}/Z(\widehat G)^{\Gamma}.
\end{eqnarray} 
Note that 
\begin{eqnarray} 
\nonumber 
C_{\tphi}(\widehat{\tG})^\circ/Z(\widehat \tG)^{\Gamma} &\s& (C_{\tphi}(\widehat{\tG})^\circ \cdot Z(\widehat \tG))/Z(\widehat \tG) \subset (\widehat{\tG})_{\ad}, \\
\nonumber 
T_{\tphi}/Z(\widehat \tG)^{\Gamma} &\s& (T_{\tphi} \cdot Z(\widehat \tG)) / Z(\widehat \tG) \subset (\widehat{\tG})_{\ad}, \\
\nonumber 
C_{\phi}(\widehat{G})^\circ/Z(\widehat G)^{\Gamma} &\s& (C_{\phi}(\widehat{G})^\circ \cdot Z(\widehat G))/Z(\widehat G) \subset (\widehat{G})_{\ad}, \\
\nonumber 
T_{\phi}/Z(\widehat G)^{\Gamma} &\s& (T_{\phi} \cdot Z(\widehat G))/Z(\widehat G) \subset (\widehat{G})_{\ad}.
\end{eqnarray} 

Let $\bG_{\scn}$ be the simply-connected cover of $\bG_{\der}.$
Consider $\phi_{\scn} \in \Phi(G_{\scn})$
\begin{equation} \label{projection b}
\phi_{\scn} = pr_{\scn} \circ \phi,
\end{equation}
where $pr_{\scn}$ is the projection  $\widehat{G} \twoheadrightarrow \widehat{G_{\scn}},$ cf., \eqref{projection}. Note that 
\[
\widehat{G_{\scn}}=(\widehat{G})_{\ad}=(\widehat{\tG})_{\ad}.
\]
We then have 
\[
\bar C_{\tphi}^\circ \subset C_{\phi_{\scn}}(\widehat{G_{\scn}})^\circ \subset \widehat{G_{\scn}} ~~\text{ and }~~
\bar C_{\phi}^\circ \subset C_{\phi_{\scn}}(\widehat{G_{\scn}})^\circ \subset \widehat{G_{\scn}}.
\]
Further, we have 
 \[
 \bar T_{\tphi} \subset (Z(\widehat{\tM}/Z(\widehat{\tG})))^\circ \subset \widehat{G_{\scn}}
 ~~\text{ and }~~
  \bar T_{\phi} \subset (Z(\widehat{M}/Z(\widehat{G})))^\circ \subset \widehat{G_{\scn}}.
 \]

\begin{pro} \label{lm for identity comp}
With above notation, we have
\[
\bar C_{\tphi}^\circ =  C_{\phi_{\scn}}(\widehat{G_{\scn}})^\circ = \bar C_{\phi}^\circ
\]
\end{pro}
\begin{proof}
It suffices to show that $\bar C_{\tphi}^\circ =  C_{\phi_{\scn}}(\widehat{G_{\scn}})^\circ,$ as the approach is applied to $\phi$ verbatim.
We first consider the following commutative diagram
\[
\begin{CD}
@.   @.1  @.  1 @.
\\
@.      @.        @VVV   @VVV  @.\\
1 @>>> Z(\widehat{\tG})^\Gamma  @>>> C_{\tphi}(\widehat{\tG})^\circ @>>> \bar C_{\tphi}^\circ := C_{\tphi}(\widehat{\tG})^\circ/ Z(\widehat{\tG})^\Gamma @>>> 1 
\\
@.      @|          @VVV   @VVV  @.
\\
1 @>>> Z(\widehat{\tG})^\Gamma @>>> C_{\tphi}(\widehat{\tG}) @>>> C_{\tphi}(\widehat{\tG})/Z(\widehat{\tG})^\Gamma @>>> 1 \\
@.      @.          @VVV   @VVV  @.\\
@.   @. \pi_0(C_{\tphi}(\widehat{\tG})) @= \pi_0(C_{\tphi}(\widehat{\tG})) @>>> 1 .
\end{CD}
\]
The right vertical exact sequence above implies that
\begin{eqnarray} \label{cphibar}
\bar C_{\tphi}^\circ &\subset& (C_{\tphi}(\widehat{\tG})/Z(\widehat{\tG})^\Gamma)^\circ \\
\nonumber
&\subset& C_{\tphi}(\widehat{\tG})/Z(\widehat{\tG})^\Gamma.
\end{eqnarray}
Note that the index 
\[
[C_{\tphi}(\widehat{\tG})/Z(\widehat{\tG})^\Gamma : \bar C_{\tphi}^\circ]
\]
is finite and so is 
\[
[(C_{\tphi}(\widehat{\tG})/Z(\widehat{\tG})^\Gamma)^\circ : \bar C_{\tphi}^\circ].
\]
Note that the quotient $\bar C_{\tphi}^\circ$ is connected, due to the isomorphism 
\[ C_{\tphi}(\widehat{\tG})^\circ/ Z(\widehat{\tG})^\Gamma \s (C_{\tphi}(\widehat{\tG})^\circ \cdot Z(\widehat{\tG}))/ Z(\widehat{\tG}).
\]  
From \eqref{cphibar}, we then have 
\begin{equation} \label{an eqaulity in iendtity comp}
\bar C_{\tphi}^\circ = (C_{\tphi}(\widehat{\tG})/Z(\widehat{\tG})^\Gamma)^\circ.
\end{equation}
We consider 
\begin{equation} \label{def of X}
X^{\tG}(\tphi) := \{ \bold a \in H^1(W_F,\widehat{\tG/G}) : \bold a \tphi \s \tphi  \},
\end{equation}
which lies in the following exact sequence (cf., \cite[Lemma 5.3.4]{chaoli} and  \cite[Lemma 4.9]{choiymulti})
\[
1 \longrightarrow C_{\tphi}(\widehat{\tG})/Z(\widehat{\tG})^\Gamma \longrightarrow C_{\phi_{\scn}}(\widehat{G_{\scn}}) \longrightarrow X^{\tG}(\tphi) \longrightarrow 1.
\]
This is due to the fact that $C_{\tphi}(\widehat{\tG})/Z(\widehat{\tG})^\Gamma \subset C_{\phi_{\scn}}(\widehat{G_{\scn}}).$
Since $X^{\tG}(\tphi)$ is finite, we have
\[
(C_{\tphi}(\widehat{\tG})/Z(\widehat{\tG})^\Gamma)^\circ = C_{\phi_{\scn}}(\widehat{G_{\scn}})^\circ.
\]
Therefore, from \eqref{an eqaulity in iendtity comp}, the proof is complete.
\end{proof}
\begin{pro} \label{a key prop}
With the above notation,
we have
\begin{equation} \label{last arg for Arthur conj}
W^{\circ}_{\tphi} = W^{\circ}_{\phi}.
\end{equation}
\end{pro}
\begin{proof}
Since we have, from an equality in \cite[p.64]{mok13},
\[
W_{\tphi}^\circ = N_{\bar C_{\tphi}^\circ}(\bar T_{\tphi})/\bar T_{\tphi} ~~\text{ and } ~~W_{\phi}^\circ = N_{\bar C_{\phi}^\circ}(\bar T_{\phi})/\bar T_{\phi} 
\]
it suffices to show that 
\[
\bar T_{\tphi} = \bar T_{\phi},
\]
due to Proposition \ref{lm for identity comp}.
Using the fact that $\tbG/\bG \s \tbM/\bM$ as $F$-tori (see Remark \ref{twoquotients}), and due to \eqref{quotients}, we have the following commutative diagram
\begin{equation} \label{important comm diag}
\begin{CD}
@.  1 @. 1 @.  1 @. 1
\\
@.      @VVV        @VVV   @VVV @VVV \\
1 @>>>  (\widehat{\tG/G})^\Gamma @>>> Z(\widehat{\tG})^\Gamma @>>> Z(\widehat{G})^\Gamma  @>>> H^1(F,  \widehat{\tG/G}) @. 
\\
@.      @|        @VVV   @VVV  @| \\
1 @>>>  ((\widehat{\tM/M})^\Gamma)^\circ @>>> (Z(\widehat{\tM})^\Gamma)^\circ @>>> (Z(\widehat{M})^\Gamma)^\circ @>>>  H^1(F,  \widehat{\tM/M})^\circ 
\\
@.      @VVV          @VVV   @VVV  @VVV
\\
1 @>>> 1 @>>>   \bar T_{\tphi} @>>> \bar T_{\phi}  @>>> 1 \\
@.      @VVV          @VVV   @VVV  @.\\
@.  1 @. 1 @. 1 @. 
\end{CD}
\end{equation}
Therefore, we $\bar T_{\tphi} = \bar T_{\phi}$ and the proof is complete.
\end{proof}

\begin{rem}
For the case of $(\tbG, \bG)=(\GSO_n, \SO_n)$ and $(\GSp_{2n}, \Sp_{2n}),$  Proposition \ref{a key prop} was discussed in \cite[Section 4.2]{xumathann}.
\end{rem}

%%%%%%%%%%%%%
\subsection{The working hypothesis} \label{sec two hypoth}

We shall assume the following hypothesis to prove our main theorems: Theorem \ref{thm from tG} applying $\bJ=\tbG,$ and Theorem \ref{thm from G} applying $\bJ=\bG.$

\begin{hyp} \label{workhyp}
Let $\boldsymbol{J}$ be a connected reductive group over $F,$ and let $\bM_{\bJ}$ be an $F$-Levi subgroup of $\boldsymbol{J}.$
Under the validity of \eqref{pre temp llc} for tempered representations of $J$ and for discrete series representations of $M_J,$ given $\phi_J \in \Phi_{\disc}(M_J)$ and $\sigma_J \in \Pi_{\phi}(M_J),$ we have
\begin{equation} \label{hyp for W0}
W_{\sigma_J}^\circ = W_{\sigma_J, \phi_J} ^\circ. 
\end{equation}
Here, $W_{\phi_J}$ is considered as a subgroup of $W_{M_J}$ (see Section \ref{rgp arthur conj}, \cite[p.45]{art89ast}).
\end{hyp}

\begin{rem} \label{question for gspin W0}
The equality \eqref{hyp for W0} is known for those cases where the Arthur's conjecture is proved (cf., \cite{art12, bg12, bggspin,  baj05,  chaoli, cgsu, chgo12, cgclassic, mok13}).
\end{rem}

\subsection{Proof of Arthur's conjectures under Working Hypothesis \ref{workhyp} for $\bJ=\tbG$} \label{sec arthur conj from tG}
%%%%%%%%%%%%%%%%%%%%%%%%%%%%%%%%%%%%%%%%%

In this section, we assume Working Hypothesis \ref{workhyp} for $\bJ=\tbG.$
Then, under the validity of \eqref{pre temp llc} for tempered representations of $J=\tG$ and for discrete series representations of $M_J=\tM,$  we construct tempered $L$-packets of $G,$ restricting representations from $\tG$ to $G$ and from $\tM$ to $M.$

As in Section \ref{a key proposition}, given $\phi \in \Phi(G),$ there is a lift $\tphi \in \Phi_{\temp}(\tG)$ such that
\[
\phi = pr \circ \tphi,
\]
where $pr$ is the projection  $\widehat{\tG} \twoheadrightarrow \widehat{G}.$ 
We note that the $L$-parameter $\phi$ lies in $\Phi_{\temp}(G),$ as $\tphi \in \Phi_{\temp}(\tG).$

Using the $L$-packet $\Pi_{\tphi}(\tG)$ for $\tphi \in \Phi_{\temp}(\tG)$ in Working Hypothesis \ref{workhyp} for $\bJ=\tbG,$ we construct an $L$-packet  $\Pi_{\phi}(G)$ for $\phi \in \Phi_{\temp}(G)$ as the set of isomorphism classes of irreducible constituents in the restriction from $\tG$ to $G:$
\[
\Pi_{\phi}(G) := \{ \sigma \hookrightarrow {\Res}^{\tG}_{G}(\ts),~~ \ts \in \Pi_{\tphi}(\tG) \} / \s.
\]
Let $\bM$ be an $F$-Levi subgroup of $\bG.$
Likewise, given $\tphi \in \Phi_{\disc}(\tM)$ and its projection $\phi \in \Phi_{\disc}(M),$
from Working Hypothesis \ref{workhyp}, we construct an $L$-packet $\Pi_{\phi}(M)$ for $\phi \in \Phi_{\disc}(M):$
\[
\Pi_{\phi}(M) := \{ \sigma \hookrightarrow {\Res}^{\tM}_{M}(\ts),~~ \ts \in \Pi_{\tphi}(\tM) \} / \s.
\]
We note that the $L$-parameter $\phi$ lies in $\Phi_{\disc}(M),$ as $\tphi \in \Phi_{\disc}(\tM).$

In what follows, we prove Arthur's conjectures for both $\tG$ and $G$ under Working Hypotheses \ref{workhyp} for $J=\tG.$ 
We begin with Lemmas.
\begin{lm}  \label{a inclusion on W}
Assume that Working Hypotheses \ref{workhyp} is valid for $\bJ=\tbG.$ We identify $W_\phi, W_{\tphi}$ with a subgroup of $W_M, W_{\tM}.$ 
With the above notation, for $\ts \in \Pi_{\tphi}(\tM)$ and $\sigma \in \Pi_{\phi}(M),$ 
we have 
\begin{equation} \label{for ts}
W(\ts) \subset W_{\tphi},
\end{equation}
and 
\begin{equation} \label{for s}
W(\sigma) \subset W_\phi
\end{equation}
\end{lm}
\begin{proof}
For \eqref{for ts}, we let $w \in W_{\tM}$ be given such that $^w\ts \s \ts.$ Note that we have 
\[
\Pi_{^w\tphi}(\tG) \cap \Pi_{\tphi}(\tG) = \{ \ts \}.
\]  
Due to the disjointness in Working Hypothesis \ref{workhyp} for $L$-packets of $\tG,$ we have 
\[
^w\tphi \s \tphi.
\]
Since the elements of $W_{\tphi}$ stabilize $\bA_{\widehat \tM}$ by definition (cf., \cite[p.60]{mok13}), it thus follows that $w$ lies in $W_{\tphi}.$

For \eqref{for s}, we let $w \in W_M$ be given such that $^w\sigma \s \sigma.$ 
Let $\sigma \in \Irr(\tM)$ be a lifting of $\sigma^\sharp$ such that $\sigma \hookrightarrow {\Res}^{\tM}_{M}(\ts)$ (see \cite[Lemma 2.3]{gk82} and \cite[Proposition 2.2]{tad92}).
Since $^w\sigma \s \sigma,$ it follows from \cite[Lemma 2.4]{gk82} and \cite[Corollary 2.5]{tad92} that 
\[
^w\sigma^\sharp \s \sigma^\sharp \chi^\sharp
\]
for some character $\chi^\sharp$ of $\tM/M.$
We then have 
\[
\sigma^\sharp \chi^\sharp \in \Pi_{\tphi \chi^\sharp}(\tM)~~\text{and}~~ ^w\sigma^\sharp \in \Pi_{^w\tphi}(\tM),
\]
where $\chi^\sharp$ is considered as a $1$-cocycle of $W_F$ in $\widehat{(\tM/M)}$ by the local class field theory (cf., \cite[Section 7]{la85} and \cite[Theorem 3.5.1]{chaoli}).
Combining the isomorphism $^w\sigma^\sharp \s \sigma^\sharp \chi^\sharp$ and the disjointness in Working Hypothesis \ref{workhyp} for $L$-packets of $\tG,$ we have   
\[
^w\tphi  \s \tphi \chi^\sharp.
\]
Through the projection $pr: \widehat{\tM} \twoheadrightarrow \widehat{M},$
it follows that  
\[
^w\phi \s \phi.
\]
Since the elements of $W_\phi$ stabilize $\bA_{\widehat M}$ by definition (cf., \cite[p.60]{mok13}), we have $w \in W_\phi.$ Thus, the proof is complete.
\end{proof}

We then have the following inclusion.
\begin{lm} \label{lm imp1}
Assume that Working Hypotheses \ref{workhyp} is valid for $\bJ=\tbG.$ We identify $W_\phi, W_{\tphi}$ with a subgroup of $W_M, W_{\tM}.$
We have 
\begin{equation} \label{equ lm imp1}
W_{\tphi} \cap W(\sigma) \subset W(\ts) 
\end{equation}
\end{lm}
\begin{proof}
Suppose that there is $w \in W_{\tphi} \cap W(\sigma)$ which does not lie in $W(\ts).$  As $w \in W(\sigma) \setminus W(\ts)$ we have 
$
\tchi \in (\tM/M)^D
$
 such that 
 \[
 {^w}\ts \s \ts \tchi \not\simeq \ts.
 \]
It then follows that
\begin{equation} \label{tchi lies 1}
\tchi \notin X^{\tM}(\ts)
\end{equation}
(refer to \eqref{def of X} for $X^{\tM}(\ts)$). 
Under the validity of \eqref{pre temp llc} for discrete series representations of $\tM$ (Working Hypothesis \ref{workhyp}), we have 
\[
{^w}\ts = \ts \tchi  \in \Pi_{{^w}\tphi} \cap \Pi_{\tphi \tchi},
\]
which yields 
\[
{^w}\tphi \s \tchi \tchi.
\]
As $w \in W_{\tphi},$ we have ${^w}\tphi \s \tphi.$ It then follows that
\begin{equation} \label{tchi lies 2}
\tchi \in X^{\tM}(\tphi).
\end{equation}
However, since the local Langlands correspondence for tori (cf. \cite[Theorem 1]{lan97olga} and \cite[Theorem, p.179]{yu09}) yields
\[
X^{\tM}(\ts) \subset X^{\tM}(\tphi),
\]
it contradicts by \eqref{tchi lies 1} and \eqref{tchi lies 2}.
We thus complete the proof.
\end{proof}

\begin{cor} \label{cor for inclusion}
Assume that Working Hypotheses \ref{workhyp} is valid for $\bJ=\tbG.$ We identify $W_\phi, W_{\tphi}$ with a subgroup of $W_M, W_{\tM}.$
We have
\begin{equation} \label{equ cor for inclusion}
W_{\phi, \sigma}^\circ \subset W(\ts).
\end{equation}
\end{cor}
\begin{proof}
This is immediate from Proposition \ref{a key prop} and Lemma \ref{lm imp1}.
\end{proof}

\begin{thm} \label{thm from tG}
Assume that Working Hypotheses \ref{workhyp} is valid for $\bJ=\tbG.$ With the above notation,
for $\ts \in \Pi_{\tphi}(\tM)$ and $\sigma \in \Pi_{\phi}(M),$ 
we have
\begin{equation}  \label{arthur conj for tG}
R_{\ts} \s R_{\tphi, \ts},
\end{equation}
and
\begin{equation}  \label{arthur conj for G}
R_{\sigma} \s R_{\phi, \sigma}.
\end{equation}
\end{thm}
\begin{proof}
Due to Lemma \ref{a inclusion on W}, we have  
\begin{equation} \label{reduction for W}
W(\sigma) = W_{\phi, \sigma} = W_\phi \cap W(\sigma)~~ \text{ and } ~~ W(\ts) = W_{\tphi, \ts} = W_{\tphi} \cap W(\ts)
\end{equation}
Since the Plancherel measure is compatible with restriction (see \cite[Proposition 2.4]{choiy1}, \cite[Lemma 2.3]{go06}, for example) and $\Phi(P, A_M) = \Phi(\tP, A_{\tM}),$
we have
\begin{equation} \label{more}
W^{\circ}_{\ts} = W^{\circ}_{\sigma}.
\end{equation}
Using \eqref{hyp for W0} for $\bJ=\tbG,$ we  have
\begin{equation} \label{reduction for W'}
W^{\circ}_{\sigma} = W^{\circ}_{\tphi, \ts}.
\end{equation}
We then have
\begin{eqnarray} \label{equal for si}
W^\circ_{\phi, \sigma} &=& W^\circ_{\phi} \cap W(\sigma) \\
\nonumber
& \overset{\eqref{last arg for Arthur conj}}{=}& W^\circ_{\tphi} \cap W(\sigma) \\
\nonumber
 &\overset{\eqref{equ lm imp1}}{=}& 
 W^\circ_{\tphi} \cap W(\ts) \cap W(\sigma) \\
 \nonumber
& =& W^\circ_{\tphi, \ts} \cap W(\sigma)  \\
\nonumber
&\overset{\eqref{reduction for W'}}{=}& W^\circ_{\sigma} \cap W(\sigma) \\
\nonumber
&=& W^\circ_\sigma.
 \end{eqnarray}

Thus, combining \eqref{reduction for W} -- \eqref{equal for si}, 
  Arthur conjectures \eqref{arthur conj for tG} and \eqref{arthur conj for G} are now proved.
\end{proof}

\subsection{Proof of Arthur's conjectures under Working Hypothesis \ref{workhyp} for $\bJ=\bG$} \label{sec arthur conj from G} 

Then, under the validity of \eqref{pre temp llc} for tempered representations of $\bJ=\bG$ and for discrete series representations of $M_J=M,$  we construct tempered $L$-packets of $\tG,$ lifting representations from $G$ to $\tG$ and from $M$ to $\tM.$

Let $\phi \in \Phi_{\temp}(G)$ be given.  As in Section \ref{sec arthur conj from tG}, there is a lift $\tphi \in \Phi_{\temp}(\tG).$ 
For each $\sigma \in \Pi_{\phi}(G),$ choose $\ts_{0, \sigma} \in \Irr_{\temp}(\tG)$ such that $\sigma \subset \Res_G^{\tG}({0, \sigma} ),$ due to \cite[Lemma 2.3]{gk82} and \cite[Proposition 2.2]{tad92}. Note if we have $\ts_1 \in  \Irr_{\temp}(\tG)$ with $\sigma \subset \Res_G^{\tG}(\ts_1),$ then $\ts_1$ must be of the form $\ts_1\s {0, \sigma}  \tchi$ for some $\tchi \in \Hom(\tG/G, \CC^{\times}).$ 
We have 
\begin{equation} \label{iso 1cocycle}
\Hom(\tG/G, \CC^\times) \s H^1(F, \widehat{\tG/G}),
\end{equation}
due to the local class field theory (or local Langlands correspondence for tori $\Hom(T, \CC^\times) \s \Phi(T),$ cf. \cite[Theorem 1]{lan97olga} and \cite[Theorem, p.179]{yu09})). We then define 
\[
X^{\tG}(\tphi) :=\{ \teta \in \Hom(\tG/G, \CC^\times):  \teta \tphi \s \tphi \}.
\]
Note that this set $X^{\tG}(\tphi)$ is a finite abelian group consisting of unitary characters and can be defined in terms of cohomological classes (refer to \cite[Remark 4.8]{} and  \cite[p.74]{chaoli}).
Using the $L$-packet $\Pi_{\phi}(G)$ for $\phi \in \Phi_{\temp}(G)$ in Working Hypothesis \ref{workhyp} for $\bJ=\bG,$ we construct an $L$-packet  $\Pi_{\tphi}(G)$ for $\tphi \in \Phi_{\temp}(\tG)$ as the set of isomorphism classes 
\[
\Pi_{\tphi}(\tG) := \{  \ts_{0, \sigma} \otimes \teta \in {\Irr}_{\temp}{\tG} : ~~ \sigma \in \Pi_{\phi}(G),~~  \teta \in X^{\tG}(\tphi)  \} / \s.
\]

\begin{rem}
For our purpose of the study of $R$-groups, we do not further discuss here, the stability of this $L$-packet via the endoscopic character identity, its internal structure, etc, nor does argue whether our $L$-packet is compatible with Xu's definition in \cite{xu15}. 
\end{rem}
\begin{lm} \label{dijoint lem}
Let $\tphi_1, \tphi_2 \in \Phi_{\temp}(\tG)$ be given. Suppose $\Pi_{\tphi_1}(\tG) \cap \Pi_{\tphi_2}(\tG) \neq \emptyset.$ Then we have 
\[
{\tphi}_1 \s {\tphi}_2 \chi^\sharp
\]
for some $ \chi^\sharp \in X^{\tG}(\tphi).$
\end{lm}
\begin{proof}
Since $\Pi_{\tphi_1}(\tG) \cap \Pi_{\tphi_2}(\tG) \neq \emptyset,$ we have 
\[
pr \circ {\tphi}_1 = pr \circ {\tphi}_2, 
\]
where $pr: \widehat{\tG} \twoheadrightarrow \widehat{G}.$ It then follows that ${\tphi}_1 \s {\tphi}_2 \chi^\sharp$
for some $1$-cocycle $\chi^\sharp$ of $W_F$ in $\widehat{(\tG/G)}.$ This completes the proof.
\end{proof}

Let $\bM$ be an $F$-Levi subgroup of $\bG.$
Given $\phi \in \Phi_{\disc}(M),$ we lift it to $\tphi \in \Phi_{\disc}(\tM),$ as in Section \ref{a key proposition}. 

From Working Hypothesis \ref{workhyp}, we construct an $L$-packet $\Pi_{\tphi}(\tM)$ for $\tphi \in \Phi_{\disc}(\tM):$
\[
\Pi_{\tphi}(\tM) := \{  \ts_{0, \sigma} \otimes \teta \in {\Irr}_{\disc}({\tM}) : ~~ \sigma \in \Pi_{\phi}(M),~~  \teta \in X^{\tM}(\tphi)  \} / \s.
\]

Now, under Working Hypotheses \ref{workhyp} for $\bJ=\bG$, we will prove the following Arthur's conjecture for $\tG$ and $G.$
\begin{thm} \label{thm from G}
Assume that Working Hypotheses \ref{workhyp} is valid for $\bJ=\bG.$
Let $\phi \in \Phi_{\disc}(M)$ be given.  Choose a lifting $\tphi \in \Phi_{\disc}(\tM)$ such that $\phi = pr \circ \tphi,$ with the projection $pr: \widehat{\tM} \twoheadrightarrow \widehat{M}.$
Given $\ts \in \Pi_{\tphi}(\tM)$ and $\sigma \in \Pi_{\phi}(M),$ 
we have
\begin{equation}  \label{arthur conj for tG 2}
R_{\ts} \s R_{\tphi, \ts},
\end{equation}
and
\begin{equation}  \label{arthur conj for G 2}
R_{\sigma} \s R_{\phi, \sigma}.
\end{equation}
\end{thm}

To this end, we begin with the following Lemma.
\begin{lm}  \label{a inclusion on W 2}
Assume that Working Hypotheses \ref{workhyp} is valid for $\bJ=\bG.$ We identify $W_\phi, W_{\tphi}$ with a subgroup of $W_M, W_{\tM}.$
We have 
\begin{equation} \label{for s 2}
W(\sigma) \subset W_\phi
\end{equation}
and
\begin{equation} \label{for ts 2}
W(\ts) \subset W_{\tphi},
\end{equation}
\end{lm}
\begin{proof}
For \eqref{for s 2}, we let $w \in W_M$ be given such that $^w\sigma \s \sigma.$ This yields $ \sigma \in \Pi_{^w\phi}(G) \cap \Pi_{\phi}(G),$ and by the disjointness in Working Hypothesis \ref{workhyp} for $L$-packets of $G$ we have  $^w\phi \s \phi.$ Thus, since the elements of $W_{\phi}$ stabilize $\bA_{\widehat M}$ by definition, it thus follows that $w$ lies in $W_{\phi}.$

For \eqref{for ts 2}, we let $w \in W_{\tM}$ be given such that $^w\ts \s \ts.$ Since $\ts \in \Pi_{^w\tphi}(\tG) \cap \Pi_{\tphi}(\tG),$ due to Lemma \ref{dijoint lem}, we have 
\[
^w\tphi \s \tphi \chi^\sharp
\]
for some $\chi^\sharp \in X^{\tM}(\tphi).$ This yields $^w\tphi \s \tphi$ and then the elements of $W_{\tphi}$ stabilize $\bA_{\widehat \tM}.$ Therefore, it follows that $w$ lies in $W_{\tphi}.$
\end{proof}

\begin{lm} \label{lm imp2}
Assume that Working Hypotheses \ref{workhyp} is valid for $\bJ=\bG.$ We identify $W_\phi, W_{\tphi}$ with a subgroup of $W_M, W_{\tM}.$
We have 
\[
W^\circ_{\tphi, \ts} = W^\circ_{\ts} 
\]
\end{lm}
\begin{proof}
The inclusion $\supset$ is immediate, since we have 
\begin{eqnarray}
\nonumber
W^\circ_{\ts}&=&W^\circ_{\si} \\
\nonumber
&\overset{ \eqref{hyp for W0} }{=}& W^\circ_{\phi, \si}  \\
\nonumber
&\subset& W^\circ_\phi \\
\nonumber
&\overset{\eqref{last arg for Arthur conj}}{=}& W^\circ_{\tphi}.
\end{eqnarray}
We verify the other inclusion $\subset$ by contradiction. Suppose that there is $W^\circ_{\tphi, \ts} \not\subset W^\circ_{\ts}.$ Then there is $w \in W^\circ_{\phi} \cap W(\ts) \setminus W(\sigma).$ This implies that $1 \neq w \in R_{\ts},$ but $w=1 \in R_{\tphi.}$ This contradicts by the fact that $R_{\ts} \subset R_{\tphi}.$ 

\end{proof}

\begin{proof}[Proof of Theorem \ref{thm from tG}]
Due to Lemma \ref{a inclusion on W 2}, we have  
\begin{equation} \label{reduction for W 2}
W(\sigma) = W_{\phi, \sigma} = W_\phi \cap W(\sigma)~~ \text{ and } ~~ W(\ts) = W_{\tphi, \ts} = W_{\tphi} \cap W(\ts)
\end{equation}
Since the Plancherel measure is compatible with restriction (see \cite[Proposition 2.4]{choiy1}, \cite[Lemma 2.3]{go06}, for example) and $\Phi(P, A_M) = \Phi(\tP, A_{\tM}),$
we have
\begin{equation} \label{more 2}
W^{\circ}_{\ts} = W^{\circ}_{\sigma},
\end{equation}
Using \eqref{hyp for W0} for $\bJ=\bG,$ we  have
\begin{equation} \label{reduction for W' 2}
W^{\circ}_{\ts} = W^{\circ}_{\phi, \sigma}.
\end{equation}
By Proposition \ref{lm for identity comp}, 
we thus have
\begin{eqnarray} \label{equal for si 2}
W^\circ_{\tphi, \ts} &= &W^\circ_{\tphi} \cap W(\ts) \\
\nonumber
&\overset{\eqref{last arg for Arthur conj}}{=} & W^\circ_{\phi} \cap W(\ts) \\
\nonumber
&\supset& W^\circ_{\phi} \cap W(\sigma) \cap W(\ts) \\
\nonumber
&=& W^\circ_{\phi, \sigma} \cap W(\ts) \\
\nonumber
&\overset{\eqref{equ cor for inclusion}}{=}& W^\circ_{\phi, \sigma} \\
\nonumber
&\overset{\eqref{reduction for W' 2}}{=} & W^\circ_{\ts}.
\end{eqnarray}

We combine \eqref{reduction for W 2} -- \eqref{equal for si 2} and thus
 Arthur conjectures \eqref{arthur conj for tG 2} and \eqref{arthur conj for G 2} are now proved.
\end{proof}

\begin{rem}
For our purpose of the study of $R$-groups, the choices of $\tphi$ and $\ts_{0, \sigma}$ are independent, i.e., we have $R_{\ts \teta} \s R_{\tphi \teta, \ts \teta}$ for any $\teta.$
\end{rem}

%%%%%%%%%%%%
\subsection{Closing remark} \label{sec clos rem}
Our result, under Working Hypothesis \ref{workhyp}, is applied to many cases such as $(\tbG,\bG)=(\GL_n,\SL_n),$ $(\GSp_{2n}, \Sp_{2n}),$ $(\GSO_n, \SO_n),$ $(\GSpin_n, \Spin_n),$ $(\U_n, \SU_n)$ and their $F$-inner forms. In Working Hypothesis \ref{workhyp}, the equality $W_{\sigma_J}^\circ = W_{\sigma_J, \phi_J} ^\circ$ (see  \eqref{hyp for W0}) suggests a uniform way to see Arthur's conjecture for $R$-groups in the setting \eqref{cond on G}. It may further confirm that $\tG$ is true if and only if Arthur's conjecture for $G.$ 
%%%%%%%%%%%%%%%%%%%%%

%%%%%%%%%%%%%%%%%%%%%%%%%%%%%%%%%%%%%%%%%%
\end{document}